\documentclass{amsart}
\usepackage{amssymb,latexsym,amsmath,amsthm,enumitem,mathrsfs,geometry}
\usepackage{fullpage}
\usepackage{fancyhdr}
\usepackage{color}
\usepackage{stmaryrd}
\usepackage{mathrsfs}
\usepackage{hyperref}
\usepackage{lineno}
\usepackage{diagbox}
\usepackage{array}
\usepackage{tikz}
\usetikzlibrary{arrows}
\usetikzlibrary{positioning}
\usepackage{float}
\usepackage{bbm}
\usepackage{cleveref}

\newtheorem*{thm*}{Theorem}

\newcommand{\beq}{\begin{equation}}
\newcommand{\eeq}{\end{equation}}

\newcommand{\li}{\mathrm{li}}

\newtheorem{theorem}{Theorem}
\newtheorem{lem}{Lemma}

\newtheorem{corollary}{Corollary}
\newtheorem{Conj}{Conjecture}

\newtheorem*{theorem*}{Theorem}
\newtheorem*{conjecture*}{Conjecture}
\definecolor{pink}{rgb}{1,.2,.6}
\definecolor{orange}{rgb}{0.7,0.3,0}
\definecolor{blue}{rgb}{.2,.6,.75}
\definecolor{green}{rgb}{.4,.7,.4}
\definecolor{purple}{RGB}{127,0,255}

\numberwithin{equation}{section}

\begin{document}

\title{The Prime Number Theorem and Pair Correlation of Zeros of the Riemann Zeta-Function}
\author[Goldston]{D. A. Goldston}
\address{Department of Mathematics and Statistics, San Jose State University}
\email{daniel.goldston@sjsu.edu}

\author[Suriajaya]{Ade Irma Suriajaya}
\address{Faculty of Mathematics, Kyushu University}
\email{adeirmasuriajaya@math.kyushu-u.ac.jp}

\date{\today}

\begin{abstract}
We prove that the error in the prime number theorem can be quantitatively improved beyond the Riemann Hypothesis bound by using versions of Montgomery's conjecture for the pair correlation of zeros of the Riemann zeta-function which are uniform in long ranges and with suitable error terms. 
\end{abstract}

\maketitle

\bigskip

\noindent
{\bf Added Remarks.} It has been brought to our attention that the main results in this paper have already been obtained in \cite{LPZ16}. We use here the function $F_\beta(x,T)$ defined in \eqref{Fbeta}, while in \cite{LPZ16} the authors use the same function with a change of variable $\beta = 1/\tau$. Their paper also has applications to other problems concerning primes. For further applications of extended pair correlation conjectures, we direct interested readers to the papers \cite{LPZ12}, \cite{LPZ16}, and \cite{LPZ17} of Languasco, Perelli, and Zaccagnini.

\bigskip

\section{Introduction}
Let $\pi(x)$ denote the number of primes less than or equal to $x$, and define $\mathcal{P}(x)$ by
\begin{equation} \label{piPNT} \pi(x) = \li(x) + \mathcal{P}(x), \qquad \text{where} \quad \li(x) := \int_2^x \frac{dt}{\log t}. \end{equation}
The prime number theorem is the assertion that $\mathcal{P}(x) = o(\li(x)) =o(\frac{x}{\log x})$ as $x\to \infty$, and we refer to $\mathcal{P}(x)$ as the error in the prime number theorem. The connection of $\mathcal{P}(x)$ with the zeros of the Riemann zeta-function is made most easily by using the von Mangoldt function $\Lambda(n)$. Writing 
\begin{equation} \label{psiPNT} \psi(x) := \sum_{n\le x} \Lambda(n) = x +\mathcal{R}(x), \end{equation}
then the prime number theorem is equivalent to $\mathcal{R}(x) =o(x)$ as $x\to \infty$. This formulation of the prime number theorem is so widely used that it is also called the prime number theorem and $\mathcal{R}(x)$ is also referred to as the error in the prime number theorem. The transition from $\mathcal{R}(x)$ to $\mathcal{P}(x)$ is easily made; here we will assume the Riemann Hypothesis (RH) and use \cite[Theorem 13.2]{MontgomeryVaughan2007}
\begin{equation}\label{RtoP} \mathcal{P}(x) = \frac{\mathcal{R}(x)}{\log x} - \frac{x^{1/2}}{\log x} + O\left(\frac{x^{1/2}}{\log^2 x}\right). \end{equation}

The best unconditional estimate for $\mathcal{R}(x)$ was obtained by Korobov and also Vinogradov in 1958. The Riemann Hypothesis is equivalent to the estimate 
\begin{equation} \label{PNTequivRH} \mathcal{R}(x) \ll x^{1/2+\epsilon}, \qquad \text{for any $\epsilon >0$.}\end{equation} 
The best known estimate for the error in the prime number theorem assuming the RH is 
due to von Koch \cite{vonKoch1901}, who in 1901 proved 
\begin{equation} \label{Koch} \mathcal{R}(x) = O\left( x^{1/2}\log^2 x\right), \end{equation}
and thus $\mathcal{P}(x) = O\left( x^{1/2}\log x\right)$.
On the other hand, Schmidt \cite{Schmidt1903} in 1903 proved
\begin{equation} \label{Schmidt} \mathcal{R}(x) = \Omega_{\pm}( x^{1/2}). \end{equation}
This last result is unconditional since if RH is false an even stronger result is true. At this point a surprising difficulty arises since when substituting \eqref{Schmidt} into \eqref{RtoP} the
constant obtained by Schmidt's method is too small to imply that there exists any value of $x$ for which $\mathcal{P}(x)>0$, see \cite[Footnote p. 93]{Ingham1932}, and this leaves open the question of whether $\li(x) > \pi(x) $ is always true.\footnote{ The definition we use here for $\li(x)$ in \eqref{piPNT} has $\li(x) < \pi(x)$ for $x=2,3,5,7$ but these exceptions do not occur for the usual definition of $\li(x)= \int_0^x dt/\log t$ used elsewhere in mathematics.} This question was answered by Littlewood in 1914, who proved
\begin{equation} \label{Little} \mathcal{R}(x) = \Omega_{\pm}( x^{1/2}\log \log\log x) \qquad \text{and}\quad \mathcal{P}(x) = \Omega_{\pm}(\frac{ x^{1/2}}{\log x}\log \log\log x).\end{equation}
During the last hundred years there has been no improvement in von Koch's upper bound \eqref{Koch} and Littlewood's lower bound \eqref{Little}.

As for the actual size of the error in the prime number theorem, on probability grounds it has been conjectured \cite[p. 484]{MontgomeryVaughan2007} that 
\begin{equation} \label{guess}\limsup_{x\to \infty}\frac{\mathcal{R}(x)}{x^{1/2} (\log\log\log x)^2}= \frac{1}{2\pi}, \qquad \liminf_{x\to \infty}\frac{\mathcal{R}(x)}{x^{1/2} (\log\log\log x)^2}= -\frac{1}{2\pi}.\end{equation}

In this paper we will improve on the RH bound \eqref{Koch} by assuming in addition conjectures related to the pair correlation for zeros of the Riemann zeta-function. The first result of this type is due to Gallagher and Mueller \cite{GaMu78} in 1978, who proved that 
\begin{equation} \label{RlittleO} \mathcal{R}(x) = o(x^{1/2}\log^2 x), \end{equation} 
subject to the RH and the conjecture that there exists a pair correlation density function. 
Using an extension of this pair correlation density function conjecture and RH, Mueller \cite{Mueller76} proved that 
\begin{equation} \int_0^X(\pi(x+\lambda\log X) -\pi(x) )^2 \, dx \sim (\lambda +\lambda^2)X,\end{equation}
a result later obtained in a different way in \cite{GoldMont}. 
In unpublished work Gallagher and Mueller were able to use this extended pair correlation density conjecture with stronger error terms to prove 
\begin{equation} \label{Rstronger} \mathcal{R}(x) = O(x^{1/2}(\log\log x)^2); \end{equation} 
however Mueller showed the conjecture is sometimes false with such strong error terms. Assuming only RH, Gallagher \cite{Gallagher80} proved that \eqref{Rstronger} holds except possibly on a set of finite logarithmic measure. 

We use here a related method introduced by Heath-Brown \cite{Heath-Brown} in 1981. For $x>0$, $T\ge 3$, and $\beta \ge 1$, define
\begin{equation} \label{Fbeta} F_\beta(x,T) := \sum_{0<\gamma,\gamma'\le T} x^{i(\gamma -\gamma')} w_\beta(\gamma -\gamma') , \qquad w_\beta(u) = \frac{4\beta^2}{4\beta^2+u^2},\end{equation} 
where the sum is over the imaginary parts $\gamma$ and $\gamma'$ of zeta-function zeros. This is a generalization of Montgomery's function $F(x,T)$ \cite{Montgomery72}, which is the case $\beta=1$ and we have $F(x,T) =F_1(x,T)$.
Heath-Brown proved \eqref{RlittleO} by assuming RH together with the conjecture that $F(x,T) = o(T\log^2T)$ holds uniformly for $T\le x \le T^M$ for any fixed number $M$. The function $F_\beta(x,T)$ was introduced by Goldston and Heath-Brown \cite{GoldH-B84} in 1984, and we will use it as the main tool in this paper.

Recently we obtained a quantitative improvement in the bound for $\mathcal{R}(x)$ as an application of a method related to a formula of Fujii \cite{GS21}. The main result we obtained is that, assuming RH and
\begin{equation} \label{Fbound} F(x,T)\ll T\log x \qquad \text{ holds uniformly for } T\le x\le T^{\log T}, \end{equation}
then this implies for $1\le h\le x$ that
\begin{equation} \label{J(x,h)} J(x,h) := \int_0^x (\psi(t+h)-\psi(t) -h)^2\, dt \ll hx\log x, \end{equation}
and this in turn implies 
\begin{equation}\label{Fujiiresult} \mathcal{R}(x) \ll x^{1/2}(\log x)^{3/2}. \end{equation}

Our approach here is to work directly with $\mathcal{R}(x)$ via its explicit formula. As we will show, conjectured bounds for $F_\beta(x,T)$ provides good upper bounds for $\mathcal{R}(x)$, and in turn conjectured asymptotic formulas for $F(x,T)$ with suitable error terms imply these bounds on $F_\beta(x,T)$. There is probably no hope of proving any of these conjectures at present, but they demonstrate that we can obtain significant improvements over \eqref{Koch} from upper bounds for exponential sums over zeros which are only powers of logarithms smaller than trivial, rather than the square root of the main term bounds needed in other problems concerning primes.

\section{Statement of Results} 

We note that $F_\beta(x,T)\ge 0$ by \eqref{FbetaIntegral} below, 
and also the trivial bound 
\begin{equation} \label{trivial} F_\beta(x,T) \ll \beta T \log^2T. \end{equation}
Anything stronger than this bound over an appropriate range will improve on \eqref{Koch}. 
\begin{Conj} \label{conj1} For $T\ge 3$, let $\mathcal{L}(T)$ and $\beta = \beta(T)$ be two non-decreasing continuous functions satisfying $1\le \beta(T) \ll \log^3T$ and $\log T \ll \mathcal{L}(T) \ll \beta(T) \log^2 T$. If $W=W(x)$ is an increasing function and $W(x)\gg (\log x)^A$ for a large constant $A$, then 
\begin{equation} \label{Conjecture1} F_\beta(x,T) \ll T \mathcal{L}(T) \qquad \text{uniformly for} \quad W(x)\ll T \ll x^{1/2}\log^2x.
\end{equation}
\end{Conj}
Define
\begin{equation} \label{M(x)} \mathcal{M}(x) := \sum_{1\le k\ll \log x} \sqrt{ \frac{\mathcal{L}(2^{k})}{\beta(2^k)}}. \end{equation}

\begin{theorem} \label{thm1} Assuming the Riemann Hypothesis and Conjecture 1, we have 
\begin{equation} \label{thm1eq} \mathcal{R}(x) \ll x^{1/2} \left(\mathcal{M}(x)+ (\log W(x))^2\right). \end{equation}
\end{theorem}

If in Theorem \ref{thm1} we use $\beta=1$ and the trivial bound \eqref{trivial}, we get the RH bound \eqref{Koch}.
The strongest results are obtained by taking $\mathcal{L}(T) = \log T$, from which we obtain the following corollaries. 
\begin{corollary} \label{cor1} Assume the Riemann Hypothesis. For $\beta=(\log T)^{3-2a}$ with fixed $0<a \le 3/2$, if 
\begin{equation} \label{cor1assume} F_{\beta}(x,T) \ll T\log T \qquad \text{ uniformly for} \quad \frac{T^2}{\log^4T}\ll x \ll T^{(\log T)^{(2-a)/a}}, \end{equation} then this implies
\begin{equation} \label{cor1eq}\mathcal{R}(x) \ll x^{1/2}\log^a x. \end{equation}
\end{corollary} 
For the case $a=3/2$ and thus $\beta=1$ in \Cref{cor1}, we can weaken \eqref{cor1assume} slightly.
\begin{corollary} \label{cor2}
Assuming the Riemann Hypothesis. If
\begin{equation} \label{cor2assume} F(x,T) \ll T\log x \qquad \text{ uniformly for} \quad \frac{T^2}{\log^4T}\ll x \ll T^{(\log T)^{1/3}}, \end{equation}
then we have
\begin{equation} \mathcal{R}(x) \ll x^{1/2}(\log x)^{3/2}. \end{equation}
\end{corollary}
This is the result \eqref{Fujiiresult} obtained in \cite{GS21} with a smaller range of validity for the conjecture.

\begin{corollary} \label{cor3}
Assuming the Riemann Hypothesis. Let $A> 1$ and 
\begin{equation} \label{betalem3} \beta = \frac{\log^3T}{A^4(\log\log 2T)^2}. \end{equation}
If
\begin{equation} \label{lem3assume} F_{\beta}(x,T) \ll T\log T, \qquad \text{ uniformly for} \quad \frac{ T^2}{ \log^4T}\ll x \ll T^{T^{1/A}}, \end{equation} 
then we have
\begin{equation} \label{lem3eq} \mathcal{R}(x) \ll A^2x^{1/2}(\log\log x)^2. \end{equation}
\end{corollary}
This is a version of Gallagher and Mueller's \eqref{Rstronger}. We have no evidence to support the conjecture \eqref{lem3assume} in such a long range of validity for $x$.

\medskip
The parameter $\beta$ is very useful in obtaining these results, but following \cite{FSZ09}, we show that it is possible to eliminate this parameter and formulate our results entirely in terms of conjectures on $F(x,T)$.

\begin{Conj} \label{conj2} For $T\ge 3$, let $\mathcal{L}(T)$ and $\beta = \beta(T)$ be two non-decreasing continuous functions satisfying $1\le \beta(T) \ll \log^3T$ and $\log T \ll \mathcal{L}(T) \ll \beta(T) \log^2 T$. If $W=W(x)$ is an increasing function and $W(x)\gg (\log x)^A$ for a large constant $A$, then 
\begin{equation} \label{Conjecture2b} \max_{ \frac{1}{2\log T} \le v \le 2\log T} \left|F(xv,T)-F(x,T)\right| \ll \frac{T\mathcal{L}(T)}{\beta(T)^2} \qquad \text{uniformly for} \quad W(x)\ll T \ll x^{1/2}\log^2x.
\end{equation}
\end{Conj}

\begin{theorem}\label{thm2} \Cref{conj2} implies \Cref{conj1} and \eqref{Conjecture1} with the same choices of $\mathcal{L}(T)$, $\beta(T)$, and $W(x)$. 
\end{theorem}

In \cite{FSZ09} it is conjectured that $F(x,T)$ satisfies an asymptotic formula with a power of logarithm savings error term. Using this we obtain the following result.
\begin{corollary} \label{cor4} Let $N(T)$ denote the number of complex zeros of the Riemann zeta-function in the upper half-plane up to a given height $T>0$ (explicitly given in \eqref{N(T)}). If for some constant $B>0$ 
\begin{equation} \label{StrongFConj} F(x,T) = N(T)\left( 1 + O\left(\frac{1}{(\log T)^B}\right)\right) \qquad \text{uniformly for} \quad W\left(\frac{x}{\log T}\right) \ll T \ll x^{1/2}\log^3x,
\end{equation}
then Conjectures 1 and 2 hold with $\mathcal{L}(T)= \log T$, $1\le \beta(T)\ll (\log T)^{B/2} $, and the same choice of $W$. 
\end{corollary}
Thus, for example, we obtain the result \eqref{lem3eq} from \eqref{StrongFConj} with $W(x)=(\log x)^A$ and $B=6$.

\section{Two Lemmas}

Our first lemma is a slight generalization of a lemma from \cite{GoldH-B84} and \cite{Heath-Brown}. 
\begin{lem}\label{Lemma1} For $x\ge 1$, $3\le s<t$, and $0<\beta =\beta(t) \le t$, let
\begin{equation} \label{calF} \mathcal{F}_\beta(x,t,s) := \max_{s\le v,v'\le t}F_{\beta(v')}(x,v).\end{equation}
Then we have 
\begin{equation} \left|\sum_{s< \gamma\le t} x^{i\gamma} \right| \ll \sqrt{\frac{t}{\beta(t)}\mathcal{F}_\beta(x,t,s)}. \end{equation}
\end{lem}

\begin{proof}[Proof of Lemma \ref{Lemma1}]
We first note that
\begin{equation} \label{FbetaIntegral} F_\beta(x,T) = \beta\int_{-\infty}^\infty e^{-2\beta |u|}\left|\sum_{0<\gamma \le T}x^{i\gamma}e^{i\gamma u}\right|^2\, du,
\end{equation}
which immediately follows from the formula
\begin{equation} \label{F_bformula} \beta \int_{-\infty}^\infty e^{-2\beta |u|} e^{ivu}\, du = w_\beta(v). \end{equation}
One form of the Gallagher-Sobolev inequality \cite[Lemma 1.1]{Montgomery71} states that for any $C^1$ function $f$ on the interval $[-a,a]$
\[ |f(0)| \le \frac{1}{2a}\int_{-a}^a |f(u)|\, du + \frac{1}{2}\int_{-a}^a |f'(u)|\, du. \]
Applying this with $a= 1/\beta(t)$, $3\le s<t $, and
\[ f(u) = \left(\sum_{s<\gamma \le t } x^{i\gamma}e^{i \gamma u}\right)^2,\]
we have, on using the Cauchy-Schwarz inequality
\[ \begin{split} \left|\sum_{s<\gamma\le t} x^{i\gamma} \right|^2 &\le \frac{\beta(t)}{2} \int_{-1/\beta(t)}^{1/\beta(t)}\left|\sum_{s<\gamma \le t } x^{i\gamma}e^{i \gamma u}\right|^2\, du + \int_{-1/\beta(t)}^{1/\beta(t)}\left|\sum_{s<\gamma \le t } x^{i\gamma}e^{i \gamma u}\right| \left|\sum_{s<\gamma \le t } \gamma x^{i\gamma}e^{i \gamma u}\right|\, du\\ & 
\ll \beta(t) \int_{-\infty}^{\infty}e^{-2\beta(t)|u|}\left|\sum_{s<\gamma \le t } x^{i\gamma}e^{i \gamma u}\right|^2\, du \\&
\qquad +\left( \int_{-\infty}^{\infty}e^{-2\beta(t)|u|} \left|\sum_{s<\gamma \le t } x^{i\gamma}e^{i \gamma u}\right|^2 \, du\right)^{1/2} \left( \int_{-\infty}^{\infty}e^{-2\beta(t)|u|} \left|\sum_{s<\gamma \le t } \gamma x^{i\gamma}e^{i \gamma u}\right|^2 \, du\right)^{1/2}.
\end{split}\]
Using the inequality $|a\pm b|^2\le 2a^2 +2b^2$, we see by \eqref{FbetaIntegral} and \eqref{calF} that 
\[ \beta(t) \int_{-\infty}^{\infty}e^{-2\beta(t)|u|}\left|\sum_{s<\gamma \le t } x^{i\gamma}e^{i \gamma u}\right|^2\, du \le 2F_{\beta(t)}(x, s) + 2F_{\beta(t)}(x, t) \le 4 \mathcal{F}_\beta(x,t,s), \]
and therefore 
\begin{equation} \label{halfway} \left|\sum_{s<\gamma\le t} x^{i\gamma} \right|^2 \ll \mathcal{F}_\beta(x,t,s) + \left(\frac1{\beta(t)} \mathcal{F}_\beta(x,t,s)\right)^{1/2} \left( \int_{-\infty}^{\infty}e^{-2\beta(t)|u|} \left|\sum_{s<\gamma \le t } \gamma x^{i\gamma}e^{i \gamma u}\right|^2 \, du\right)^{1/2}.
\end{equation}
By partial summation with 
\begin{equation} \label{S(r)} S(r) := \sum_{s<\gamma\le r} (xe^u)^{i\gamma}, \end{equation}
we have 
\[\begin{split} \sum_{s<\gamma \le t } \gamma x^{i\gamma}e^{i \gamma u} &= \int_{s}^t r dS(r) = tS(t) - \int_{s}^t S(r)\, dr \\ & = t\sum_{s<\gamma \le t } x^{i\gamma}e^{i \gamma u} - \int_s^{t} \sum_{s<\gamma \le r } x^{i\gamma}e^{i \gamma u}\, dr. \end{split}\]
By the same argument used to obtain \eqref{halfway}, as well as \eqref{FbetaIntegral} and the Cauchy-Schwarz inequality, we have
\[ \begin{split} \int_{-\infty}^{\infty}&e^{-2\beta(t)|u|} \left|\sum_{s<\gamma \le t} \gamma x^{i\gamma}e^{i \gamma u}\right|^2 \, du \\ &\ll t^2 \int_{-\infty}^{\infty}e^{-2\beta(t)|u|} \left|\sum_{s<\gamma \le t } x^{i\gamma}e^{i \gamma u}\right|^2 \, du + \int_{-\infty}^{\infty}e^{-2\beta(t)|u|} \left|\int_s^{t} \sum_{s<\gamma \le r } x^{i\gamma}e^{i \gamma u}\, dr\right|^2 \, du \\ & \ll \frac{t^2}{\beta(t)} \mathcal{F}_\beta(x,t,s) + (t-s) \int_s^{t} \left(\int_{-\infty}^{\infty}e^{-2\beta(t)|u|} \left|\sum_{s<\gamma \le r} x^{i\gamma}e^{i \gamma u}\right|^2 \,du \right)\,dr \\&
\ll \frac{t^2}{\beta(t)} \mathcal{F}_\beta(x,t,s).
\end{split}\]
Substituting this into \eqref{halfway} gives
\[ \left|\sum_{s<\gamma\le t} x^{i\gamma} \right|^2 \ll \left( 1 + \frac{t}{\beta(t)} \right)\mathcal{F}_\beta(x,t,s)\ll
\frac{t}{\beta(t)} \mathcal{F}_\beta(x,t,s)\]
if $0< \beta(t) \le t$. Lemma \ref{Lemma1} now follows. 
\end{proof}

Our next lemma is a formula from \cite{FSZ09} for evaluating $F_\beta(x,T)$ using $F(x,T)$.
\begin{lem}\label{Lemma2} We have, for $x\ge 1$, $T\ge 3$, and $\beta >0$, 
\begin{equation} \label{lemma1eq} F_\beta(x,T) = F(x,T) + \beta(1-\beta^2) \int_{-\infty}^\infty (F(xe^u,T)- F(x,T)) e^{-2\beta |u|}\, du. \end{equation}
\end{lem}

\begin{proof}[Proof of Lemma \ref{Lemma2}] This can be easily verified directly using \eqref{FbetaIntegral} and \eqref{F_bformula}.
It can also be obtained from these equations using the algebra identity
\[ w_\beta(a) - \beta^2w(a) = (1-\beta^2)w(a)w_\beta(a), \]
which implies 
\[ w_\beta(a) = w(a) + (1-\beta^2)\left(w(a)w_\beta(a) - w(a)\right). \]

\end{proof}

\section{Proof of the theorems and corollaries}

The Riemann-von Mangoldt formula states
\begin{equation} \label{N(T)} N(T) := \sum_{0 < \gamma \le T} 1= \frac{T}{2\pi} \log \frac{T}{2 \pi} - \frac{T}{2\pi} +O(\log T) ,\end{equation}
see for example \cite[Theorem 25]{Ingham1932} or \cite[Theorem 9.4]{Titchmarsh}.
Thus $N(T) \sim \frac{T}{2\pi}\log T$, and we also obtain
\begin{equation} \label{N(T+1)-N(T)} N(T+1) -N(T) = \sum_{T<\gamma \le T+1} 1 \ll \log T, \end{equation}
(see \cite[Theorem 25a]{Ingham1932} or \cite[Theorem 9.2]{Titchmarsh}).
\begin{proof}[Proof of Theorem 1] The truncated explicit formula for $\psi(x)$ \cite[Theorem 12.5]{MontgomeryVaughan2007} implies, for $x\ge 2$ and $Y\ge 3$,
\[ \psi(x) = x - \sum_{\substack{\rho\\ |\gamma|\le Y}} \frac{x^\rho}{\rho} + O\left(\frac{x}{Y}(\log xY)^2\right) + O(\log x).\]
We take $3\le W < Y$, where
\[ Y= 3 x^{1/2}\log^2(2x).\]
Assuming RH, then the complex zeros of the zeta function come in complex conjugate pairs $\rho = 1/2+i\gamma$ and $\overline{\rho} = 1/2-i \gamma$, and we have
\[ \frac{\mathcal{R}(x)}{x^{1/2}} = -2\,{\rm Im}\sum_{W<\gamma\le Y }\frac{x^{i\gamma}}{\gamma} +O\left(\sum_{0<\gamma \le W} \frac1{\gamma}\right)+O\left(\sum_{0<\gamma \le Y} \frac1{\gamma^2}\right) + O(1).\]
Since by \eqref{N(T+1)-N(T)} we have $\sum_{0<\gamma \le W} \frac1{\gamma} \ll \log^2W$ and $\sum_{0<\gamma \le Y} \frac1{\gamma^2}\ll 1$, see \cite[Theorem 25b]{Ingham1932}, we conclude
\begin{equation}\label{Restimate1}\left|\frac{\mathcal{R}(x)}{x^{1/2}}\right| \le 2 \left|\sum_{W<\gamma\le T }\frac{x^{i\gamma}}{\gamma}\right| +O(\log^2W).\end{equation}
By partial summation using $S(r)= \sum_{s<\gamma\le r} x^{i\gamma}$ from \eqref{S(r)} with $u=0$, we have 
\[ \sum_{s<\gamma\le t }\frac{x^{i\gamma}}{\gamma} = \int_{s}^{t}\frac1{r}dS(r) 
= \frac{S(t)}{t} + \int_s^{t} \frac{S(r)}{r^2}\, dr, \]
and hence
\[\left|\sum_{s<\gamma\le t }\frac{x^{i\gamma}}{\gamma}\right| \ll \frac1{s}\max_{s\le v\le t}\left|\sum_{s<\gamma\le v} x^{i\gamma}\right|.\]
Taking $s=y$ and $t=2y$, and applying Lemma 1 with Conjecture 1, we conclude
\begin{equation}\label{neardone}\left|\sum_{y<\gamma\le 2y }\frac{x^{i\gamma}}{\gamma}\right| \ll \frac1{y}\sqrt{\frac{y}{\beta(2y)}\mathcal{F}_\beta(x,2y,y)}\ll \sqrt{ \frac{\mathcal{L}(2y)}{\beta(2y)}}.\end{equation}

Taking $y= 2^{k-1}$ we have 
\[\left|\sum_{W<\gamma\le Y }\frac{x^{i\gamma}}{\gamma}\right| \le \sum_{k_1\le k\le k_2} \left|\sum_{2^{k-1}<\gamma\le 2^{k}} \frac{x^{i\gamma}}{\gamma}\right|,\] 
where $k_1$ and $k_2$ are chosen so that $2^{k_1 -1} < W(x)\le 2^{k_1}$ and $2^{k_2-1} <Y\le 2^{k_2}$. 
Using \eqref{neardone} we have
\[ \left|\sum_{W<\gamma\le Y }\frac{x^{i\gamma}}{\gamma}\right|
\ll \sum_{1\le k\ll \log x} \sqrt{ \frac{\mathcal{L}(2^{k})}{\beta(2^{k})}} := \mathcal{M}(x),
\]
where we have applied Conjecture 1 for $F_\beta(x, u)$ with $W\ll u \ll x^{1/2}\log^2x$ as required, and we have then insignificantly increased the upper bound by extending the range of $k$.
Theorem 1 now follows from \eqref{Restimate1}. 
\end{proof}

\begin{proof}[Proof of \Cref{cor1}]
For fixed $0< a\le 3/2$, take in Conjecture \ref{conj1}, $\beta=\beta(T)=(\log{T})^{3-2a}$, $\mathcal{L}(T)=\log{T}$, and $W(x) = e^{(\log x)^{a/2}}$. Thus
\[ F_\beta(x,T) \ll T\log T \qquad \text{uniformly for} \qquad \frac{T^2}{\log^4 T}\ll x \ll T^{(\log T)^{(2-a)/a}},\]
and we have
\[ \mathcal{M}(x) \ll \sum_{1\le k \ll \log x}\sqrt{\frac{k}{k^{3-2a}}}\ll \log^ax. \]
Applying Theorem \ref{thm1} we obtain \Cref{cor1}.
\end{proof}

\begin{proof}[Proof of \Cref{cor2}]
In the case when $\beta=1$ we replace Conjecture 1 by \eqref{cor2assume} in Theorem 1 where the range where this conjecture holds is equivalent to taking $W(x)= e^{(\log x)^{3/4}}$ in \eqref{Conjecture1}. Then \eqref{neardone} becomes 
\[ \left|\sum_{y<\gamma\le 2y }\frac{x^{i\gamma}}{\gamma}\right| \ll \sqrt{\log x}. \]
Taking $y=2^{k-1}$ as we did below \eqref{neardone}, we apply this bound $\ll \log x$ times and obtain from \eqref{Restimate1} 
\[\mathcal{R}(x) \ll x^{1/2}\Big( (\log x)^{3/2} + \log^2W \Big)\] and the result follows.
\end{proof}

\begin{proof}[Proof of \Cref{cor3}]
Take $W(x) = \log^Ax$. Then by Theorem \ref{thm1} we will obtain \eqref{lem3eq} if
$$ \mathcal{M}(x)\ll A^2(\log\log x)^2. $$
Thus we need $\sqrt{\mathcal{L}(2^k)/\beta(2^k)} \ll A^2(\log 2k)/k$ which we obtain with $\mathcal{L}(T) = \log T$ by choosing $\beta(T)$ as in \eqref{betalem3}.
\end{proof}

\begin{proof}[Proof of \Cref{thm2}] By \Cref{Lemma2}, for $\beta\ge 1$
\begin{equation} \label{thm2Step1}
|F_\beta(x,T)-F(x,T)| \ll \beta^3\left( \int_{0}^V + \int_V^\infty\right) |(F(xe^{\pm u},T)- F(x,T)| e^{-2\beta u}\, du = I_1 +I_2, \end{equation}
where we take 
\[ V= \frac{\log (\beta\log T)}{\beta}. \]

Using the trivial bound $F(x,T) \ll T\log^2T$ from \eqref{trivial},
we have 
\[ I_2 \ll \beta^3 T\log^2 T \int_V^\infty e^{-2\beta u}\, du = \frac{\beta^2}2 T\log^2Te^{-2\beta V} = \frac12 T,\]
which is acceptable and holds for all $x$. 
Next, for $I_1$ we note that for $0\le u\le V$ we have $e^{-V} \le e^{\pm u} \le e^{V}$, and hence 
\[ I_1 \ll \beta^3 \max_{ e^{-V}\le v \le e^V}|F(xv,T)-F(x,T)| \int_0^\infty e^{-2\beta u}\, du \ll \beta^2\max_{ e^{-V}\le v \ll e^V}|F(xv,T)-F(x,T)|.\]
Letting $f(x) = \frac{\log x}{x}$, by calculus we note for $x >0$ that $f(x) \le f(e) = \frac{1}{e}$. Thus for $\beta\ge 1$
\[ e^V = (\beta \log T)^{\frac1{\beta}}\le e^{f(\beta)}\log T\le e^{\frac{1}{e}}\log T < 2 \log T,\]
and similarly $e^{-V}\ge e^{-\frac{1}{e}}\frac{1}{\log T} > \frac{1}{2\log T}.$
Thus by \Cref{conj2} we have $I_1\ll T\mathcal{L}(T)$ and \Cref{thm2} follows.
\end{proof}
\begin{proof}[Proof of \Cref{cor4}] We apply \eqref{StrongFConj} with $x$ replaced with $xv'$, where $v'$ is a value with $1/\log T\ll v' \ll \log T$. Then
$|F(xv',T)-F(x,T)| \ll T\log T/(\log T)^B$, and the bound in \eqref{Conjecture2b} holds if $1\le \beta(T) \ll (\log T)^{B/2}$. We have obtained this bound using \eqref{StrongFConj} in the range $W(xv'/\log T)\ll T\ll (xv')^{1/2}\log^3(xv')$, and since $W(xv'/\log T)\ll W(x)$ and $x^{1/2}\log^2x\ll (xv')^{1/2}\log^3(xv')$, we see \eqref{Conjecture2b} holds in the stated range $W(x) \ll T\ll x^{1/2}\log^2x$.
\end{proof}

\section*{Acknowledgement}

The second author was supported by JSPS KAKENHI Grant Numbers 18K13400 and 22K13895, and also MEXT Initiative for Realizing Diversity in the Research Environment.

\section*{Conflict of Interest}
We have no conflicts of interest to disclose.

\section*{Data Availability Statements}
Data sharing is not applicable to this article as no datasets were generated or analyzed during the current study.


\end{document}